\documentclass[11pt,twoside,a4paper]{article}
\usepackage[italian,english]{babel}
\usepackage{amsmath,amsfonts,amssymb,amscd,lscape,amstext,amsthm}
\usepackage{tabularx}
\newtheorem{lemma}{\rm \indent LEMMA}[section]
\newtheorem{definition}{\rm \indent DEFINITION}[section]
\newtheorem{theorem}[lemma]{\rm \indent THEOREM}

\newtheorem{proposition}[lemma]{\rm \indent PROPOSITION}
\newtheorem{corollary}[lemma]{\rm \indent COROLLARY}
\newcommand{\eps}{\varepsilon}

\newcommand{\R}{\mathbb{R}}
\renewcommand{\le }{\leqslant }
\renewcommand{\ge }{\geqslant }

\newcommand{\G}{\Gamma}

\numberwithin{equation}{section}
\begin{document}
%\title{\textbf{The local Calder\`{o}n problem and the determination
%at the boundary of the conductivity}}
%\thanks{Work supported in part by MURST, grant no. 2006014115}}
%\author{\textsc{Giovanni Alessandrini}\thanks{Dipartimento di Matematica e Informatica, Universit\`{a} degli
%Studi di Trieste, Via Valerio 12/b, 34127 Trieste,
%alessang$@$units.it} and \textsc{Romina Gaburro}\thanks{Department
%of Mathematics and Statistics, University of Limerick, Castletroy,
%Limerick, Ireland, romina.gaburro$@$ul.ie}}
\title{\bf{ SMOOTHNESS DEPENDENT STABILITY IN CORROSION DETECTION}}
\author{EVA SINCICH\thanks{Laboratory for Multiphase Processes, University
of Nova Gorica, Slovenia. Email: eva.sincich@ung.si}}

\date{}
\maketitle \footnotesize{\textbf{Abstract.} 
We consider the stability issue for the determination of a linear corrosion in a conductor by a single electrostatic measurement. We established a global \emph{log-log} type stability when the corroded boundary is simply Lipschitz. We also improve such a result obtaining a global \emph{log} stability by assuming that the damaged boundary is $C^{1,1}$-smooth.

\section{\normalsize{Introduction}}\label{sec1}
\setcounter{equation}{0} 

In this paper we study the stable determination of a corrosion coefficient on an inaccessible boundary by means of electrostatic measurements.

More precisely, we consider

\begin{equation}\label{P}
\left\{
\begin{array}
{lcl}
\Delta u=0\ ,& \mbox{in $\Omega$ ,}
\\
\dfrac{\partial u}{\partial \nu}=g\ ,   & \mbox{on $\G_A$ ,}
\\
\dfrac{\partial u}{\partial \nu} + \gamma u =0 \ , & \mbox{on $\G_I$ ,}
\end{array}
\right.
\end{equation}
where $\Gamma_A$ and $\Gamma_I$ are two open, disjoint portions of $\partial \Omega$ such that $\partial \Omega = \overline{\Gamma_A}\cup\overline{\Gamma_I}$ and $\Omega\subset \mathbb{R}^n, \ n\ge 2$. The portion $\Gamma_A$ corresponds to the part of boundary which is accessible to measurements while $\Gamma_I$ is the portion which is out of reach and where the corrosion damage occurs. The function $\gamma$ is known as corrosion coefficient and it models the surface impedance of the conductor. The inverse problem we address here consists in the determination of such $\gamma$ by means of the current density $g$ prescribed on $\Gamma_A$ and the corresponding measured potential $u|_{\Gamma_A}$.  In particular, we are interested in providing \emph{global} stability estimates for $\gamma$, or namely avoiding the \emph{a priori} hypothesis that the unknown corrosion coefficient is a small perturbation of a given and known one.  

%This note is mainly divided in two parts which can be summarized as follows. Our fist 
Our first aim is to investigate the continuous dependence of $\gamma$ upon the data when the corroded boundary $\Gamma_I$ is merely \emph{Lipschitz}. To this purpose, we notice that by the impedance condition in \eqref{P} we can formally compute $\gamma$ as 
\begin{eqnarray}\label{quoziente}
\gamma(x)= -\frac{1}{u(x)}\frac{\partial u(x)}{\partial \nu} \ . 
\end{eqnarray}
Since the potential $u$ may vanish in some points on $\Gamma_I$, it follows that the above quotient may be highly unstable. In this respect it is necessary to compute the local vanishing rate of $u$ on $\Gamma_I$. Indeed, we proved that such a rate can be controlled in an exponential manner as follows 
\begin{eqnarray}
\int_{\Delta_r(x_0)}u^2\ge \exp(-K r^{-K})
\end{eqnarray}
where $K>0$ and $\Delta_r(x_0)=B_r(x_0)\cap \Gamma_I$ with $x_0\in \Gamma_I^{2r}\subset \Gamma_I$ (see Subsection \ref{def} for a precise definition) for sufficiently small radius $r$ (see Subsection \ref{lb}).
By combining such a control with a logarithmic stability estimate for the underlying Cauchy problem we are able to prove a \emph{global} stability estimate for $\gamma$ with a \emph{log-log} type  modulus of continuity.  

The second purpose of this paper is to strengthen the hypothesis on the corroded boundary assuming that $\Gamma_I$ is \emph{$C^{1,1}$-smooth} in order to obtain a better rate of stability. Indeed under such additional a priori hypothesis, we derive a surface doubling inequality of this sort for sufficiently small radius $r$ (see Subsection \ref{cb}).
\begin{eqnarray}\label{di}
\int_{\Delta_{2r}(x_0)}u^2\le \emph{const.}\int_{\Delta_{r}(x_0)}u^2\ ,
\end{eqnarray}
which allows us to deduce that the vanishing rate of $u$ at the boundary is at most polynomial, that is 
\begin{eqnarray}
\int_{\Delta_r(x_0)}u^2\ge \frac{1}{K}r^{K} \  , 
\end{eqnarray}
for sufficiently small radius $r$ (see Subsection \ref{cb}).
Again by gathering a logarithmic stability estimate for the Cauchy problem and the above vanishing rate we provide a \emph{global} stability estimate for $\gamma$ with a \emph{single} \emph{log}. 

In addition we also give an alternative proof of the above mentioned \emph{global logarithmic} stability estimate. Such an alternative argument mostly relies on the application of the theory of the Muckehhoupt weights which justifies the computation in \eqref{quoziente} in the $L^{\frac{2}{p-1}}$ sense for some $p>1$.

Indeed, such a dependence of the modulus of continuity upon the smoothness of the boundary have been already observed in other contexts. In \cite{ABRV}, inverse problems for the determination of unknown defects with Dirichlet and Neumann condition have been studied. The authors proved that when the unknown boundary is smooth enough and hence a doubling inequality at the boundary is available then stability turns out to be of \emph{logarithmic} type. On the contrary relaxing the regularity assumptions on the unknown domain the rate of stability degenerates into a \emph{log-log} type one. 

Let us mention that global stability estimates for unknown boundary impedance coefficients have been previously discussed under analogous boundary smoothness assumptions in \cite{asv} and \cite{S} for an inverse acoustic scattering problem.

The present inverse problem has been studied in \cite{ADR} and in \cite{CFJL} in a two dimensional setting where the authors provided a \emph{global logarithmic} stability estimate for the corrosion coefficient for $C^{1,\alpha}$ corroded boundary. 

Similar inverse problems have been studied for the heat equation \cite{BCC} and for the Stokes equations \cite{BEG}, where \emph{logarithmic} stability estimates for the Robin coefficient $\gamma$ have been provided. However in such papers the analysis on the local vanishing control of the solution has not been carried over and as a consequence the stability results are stated \emph{only} on a compact set where the solution does not vanish.  

The paper is organized as follows. In Section \ref{sec2} we introduce notation and definition, the main assumptions and we state our main results in Theorem \ref{stabrough} and in Theorem \ref{stabimproved}. In Section \ref{sec4} we preliminary analyse the direct problem recalling some regularity properties of the solution in Lemma \ref{regolaritarough} and Lemma \ref{regolaritaimproved}. Moreover, in Theorem \ref{h1} we provide an a priori bound of the boundary trace of the solution in the $H^1$ norm. The proof of such a bound relies on the well-known Rellich's identity. In Subsection \ref{lb} we discuss the inverse problem under the a priori hypothesis of a merely Lipschitz boundary. In Theorem \ref{stabcp} we recall a known stability result for the underlying Cauchy problem based on unique continuation tools, while in Corollary \ref{cor} we use the latter result in order to deduce the stability for negative norms of the normal derivative of $u$. In Theorem \ref{lipprop} we provide a lower bound on the local vanishing rate of the solution $u$. The main ingredient of the proof is the so called Lipschitz Propagation of Smallness, see also \cite{ABRV, MR2}.  Finally in Proposition \ref{weight} we state a weighted interpolation inequality which was previously introduced in \cite{asv} and we conclude by giving the proof of Theorem \ref{stabrough}. In Subsection \ref{cb} we treat the inverse problem under the further $C^{1,1}$ a priori smoothness assumption on $\Gamma_I$. In Theorem \ref{stabcp2} we recall a stability result for the Dirichlet trace of the solution in $C^1$ norm. The increased smoothness regularity hypothesis on $\Gamma_1$  allows us to refine the analysis on the local vanishing control of the solution, indeed in Proposition \ref{di} a surface doubling inequality is provided. 
We use such an inequality as tool to state in Theorem \ref{lipprop2} the polynomial rate of decay of the solution at the boundary. The main argument of this proof again relies on Lipschitz Propagation of Smallness estimates, see also \cite{AMR2}. In Proposition \ref{weight2} we state a weighted interpolation inequality for a weight satisfying a polynomial vanishing rate. We conclude by giving a proof of Theorem \ref{stabimproved}. As already mentioned, we also provide another way to obtain the logarithmic stability results which involves in Proposition \ref{mw} the notion of Muckenhoupt weights \cite{CoFe}. We complete Section \ref{sec5} with an alternative proof of the Theorem \ref{stabimproved} relying on the result achieved in Proposition \ref{mw}.

\section{\normalsize{Main Results}}\label{sec2}
\setcounter{equation}{0}
\subsection{\normalsize{Notation and definitions}}\label{def}
We introduce some notation that we shall use in the sequel.

For any $x_0\in \partial \Omega$ and for any $\rho>0$ we shall denote
\begin{eqnarray}
%&&\Gamma_{\rho}(x_0)=\Omega \cap B_{\rho}(x_0)\\
%&&\Delta_{\rho}(x_0)={\overline \Gamma_{\rho}}(x_0)\cap \partial \Omega \\
&&\Gamma_A^{\rho}=\{x\in \Gamma_A : \mbox{dist}(x,\Gamma_I)>\rho\}\ ,\\
&&\Gamma_I^{\rho}=\{x\in \Gamma_I : \mbox{dist}(x,\Gamma_A)>\rho\}\ ,\\
%&&U_A^{\rho}=\{x \in \overline{\Omega} : \mbox{dist}(x,\Gamma^{\rho}_A)<\rho \} \\
%&&U_I^{\rho}=\{x\in \overline{\Omega} : \mbox{dist}(x,\Gamma^{\rho}_I)<\rho\}\\
%&&\Omega_{\rho}=\{x\in\Omega : \mbox{dist}(x,\partial \Omega)>\rho \}\\
&&\Gamma_{\rho}(x_0)=B_{\rho}(x_0)\cap \overline{\Omega}\ ,\\
&&\Delta_{\rho}(x_0)=B_{\rho}(x_0)\cap \partial{\Omega}\ .
\end{eqnarray}

\begin{definition}
We shall say that a domain $\Omega$ is of \emph{Lipschitz class with constants} $r_0, M>0$ 
if for any $P \in \partial\Omega$, there exists a rigid transformation of coordinates under which we have $P=0$ and 
\begin{eqnarray}\label{chap2:1l}
 \Omega \cap B_{r_0}=\{(x',x_n): x_n>\varphi(x')\}\ 
\end{eqnarray}
where
\begin{eqnarray}\label{chap2:2l}
\varphi:B^{'}_{r_0}\subset \R^{n-1} \rightarrow \mathbb{R} 
\end{eqnarray}
 is a  Lipschitz function satisfying 
\begin{eqnarray}\label{chap2:25l}
|\varphi(0)|=|\nabla \varphi(0)|=0\ \ \mbox{and}\ \ \|\varphi\|_{C^{0,1}(B^{'}_{r_0})}\le Mr_0\ ,
\end{eqnarray}
where we denote by
$$\|\varphi\|_{C^{0,1}( B'_{r_0}(x_0))}=\|\varphi\|_{L^{\infty}(
B'_{r_0}(x_0))} +\ r_0\!\!\!\!\!\!\!\!\!\sup_{\substack {x,y  \in
B'_{r_0}(z_0)\\x\ne y }}
\frac{|\varphi (x)-\varphi(y)|}{|x-y|} _{\ \ \  } $$
and $B'_{r_0}(x_0)$ denotes a ball in $\R^{n-1}$.

\end{definition}

\begin{definition}
Given $ \alpha,\ 0<\alpha\le 1$, we shall say that a domain $\Omega$ is of \emph{class $C^{1,\alpha}$ with constants} $r_0, M>0$ 
if for any $P \in \partial \Omega$, there exists a rigid transformation of coordinates under which we have $P=0$ and 
\begin{eqnarray}\label{chap2:1}
 \Omega \cap B_{r_0}=\{(x',x_n): x_n>\varphi(x')\}\ 
\end{eqnarray}
where
\begin{eqnarray}\label{chap2:2}
\varphi:B^{'}_{r_0}\subset \R^{n-1} \rightarrow \mathbb{R} 
\end{eqnarray}
 is a  $C^{1,\alpha}$ function satisfying 
\begin{eqnarray}\label{chap2:25}
|\varphi(0)|=|\nabla \varphi(0)|=0\ \ \mbox{and}\ \ \|\varphi\|_{C^{1,\alpha}(B^{'}_{r_0})}\le Mr_0\ ,
\end{eqnarray}
where we denote
\begin{eqnarray}\label{chap2:3}
\|\varphi\|_{C^{1, \alpha}( B^{'}_{r_0})}&=&\|\varphi\|_{L^{\infty}( B^{'}_{r_0})}+ r_0\|\nabla\varphi\|_{L^{\infty}( B^{'}_{r_0})} + \\
&+&{r_0}^{1+\alpha}\sup_{\substack {x,y  \in B^{'}_{r_0}\\x\ne y }}\frac{|\nabla\varphi (x)-\nabla \varphi (y)|}{|x-y|^{\alpha}}\  .\nonumber
\end{eqnarray}
\end{definition}
\
\subsection{\normalsize{Assumptions and a-priori information}}

{\bf{Assumption on the domain}}

Given $r_0,M>0$ constants, we assume that $\Omega \subset \R^{n}$ and 
\begin{eqnarray}\label{domain}
\Omega \   \mbox{is of} \ \ \mbox{Lipschitz class with constants}\ r_0, M.
\end{eqnarray} 
%Furthermore, we assume that

%\begin{eqnarray}\label{smooth}
%\Gamma_I\ \ \mbox{is of} \  C^{1,1}\ \mbox{class with constants}\ r_0, M.
%\end{eqnarray} 
Moreover, we assume that 
\begin{eqnarray}\label{diameter}
\mbox{the diameter of}\ \Omega \ \mbox{is bounded by}\ d_0\ .
\end{eqnarray} 
{\bf{Assumption on $\gamma$}}

Given $\gamma_0>0$ constant we assume that the Robin coefficient $\gamma\ge 0$ is such that  $\mbox{supp}\ \gamma \subset \Gamma_I$ and 
\begin{eqnarray}\label{boundgamma}
\|\gamma\|_{C^{0, 1}(\Gamma_I)}\le \gamma_0 \ .
\end{eqnarray}
{\bf{Assumption on $g$}}

Given $E, \hat{r}$ positive constants we assume that the current flux $g$ is such that $\mbox{supp} \ g\subset \Gamma_A^{\hat{r}}$ and 
\begin{eqnarray}\label{boundg}
\|g\|_{C^{0, \alpha}(\Gamma_A)}\le E \ .
\end{eqnarray}

From now on we shall refer to the a-priori data as the following set of quantities $r_0, M, d_0, \gamma_0, E, \hat{r}$.

In the sequel we shall denote with $\eta(t)$ a positive increasing concave
function defined on $(0, +\infty)$, that satisfies
\begin{eqnarray}\label{eta}
&&\eta(t)\le C(\log (t))^{-\vartheta},\ \ \ \mbox{for
every}\ \ 0<t<1 \ ,
\end{eqnarray}
where
$C>0,\vartheta>0$ are constants depending on the \emph{a priori
data} only.

Let us fix an open connected portion $\Gamma$ of the boundary of $\Omega$. We introduce the trace space $H_{00}^{\frac{1}{2}}(\Gamma)$  as the interpolation space $[H^{1}_0(\Gamma),L^2(\Gamma)]_{\frac{1}{2}}$, we refer to \cite[Chap.1]{LiMa} for further details . The functions in $H_{00}^{\frac{1}{2}}(\Gamma)$ might be also characterized as the elements in $H^{\frac{1}{2}}(\partial \Omega)$ which are identically zero outside $\Gamma$ , this identification shall be understood throughout. We denote with $H_{00}^{-\frac{1}{2}}(\Gamma)$ its dual space, which also can be interpreted as a subspace of $H^{-\frac{1}{2}}(\partial \Omega)$.

\subsection{\normalsize{The main results}}

\begin{theorem}\label{stabrough}
Let $\Omega$ be a Lipschitz domain and let $\gamma_1,\gamma_2$ satisfy \eqref{boundgamma}. Let $u_i, i=1,2$ be the weak solution to the problem \eqref{P} with $\gamma=\gamma_i$ respectively. If for some $\eps$, we have 
\begin{eqnarray}\label{err}
 \|u_1 -u_2\|_{L^2(\Gamma_A)}\le \eps
\end{eqnarray}
then 
\begin{eqnarray}
\|\gamma_1 -\gamma_2 \|_{L^{\infty}(\Gamma_I^{r_0})}\le \eta \circ \eta (\epsilon)
\end{eqnarray}
\end{theorem}

\begin{theorem}\label{stabimproved}
Let $\Omega$ be a $C^{1,\alpha}$ domain with $0<\alpha\le 1$ and let $\gamma_1,\gamma_2$ satisfy \eqref{boundgamma}. Furthermore, we assume that $\Gamma_I$ is of class $C^{1,1}$ with constants $r_0, M$. Let $u_i, i=1,2$ be the weak solution to the problem \eqref{P} with $\gamma=\gamma_i$ respectively. If for some $\eps$, we have 
\begin{eqnarray}
 \|u_1 -u_2\|_{L^2(\Gamma_A)}\le \eps
\end{eqnarray}
then 
\begin{eqnarray}
\|\gamma_1 -\gamma_2 \|_{L^{\infty}(\Gamma_I^{r_0})}\le\eta (\epsilon)
\end{eqnarray}
\end{theorem}

\section{\normalsize{The direct problem}}\label{sec4}
\begin{lemma}\label{regolaritarough}
Let $\Omega$ be a Lipschitz domain.
Let $u\in H^1(\Omega)$ be a solution to \eqref{P} with $\gamma$ and $g$ satisfying the a-priori assumptions stated above. Then there exists a constant $0<\alpha<1$ and a constant $C>0$ depending on the a-priori data only, such that $u\in C^{\alpha}(\bar{\Omega})$ , such that 
\begin{eqnarray}\label{holderrough}
\|u\|_{C^{\alpha}(\bar{\Omega})}\le C \ .
\end{eqnarray}
\end{lemma}
\begin{proof}
This is a standard regularity estimate up to the boundary. The Moser iteration techniques  \cite[Theorem 8.18]{gt} fits to this task. More details ban be found in \cite{S1}. Such arguments only require the Lipschitz regularity of $\partial \Omega$.

\end{proof}

\begin{lemma}\label{regolaritaimproved}
Let $\Omega$ be a $C^{1,\alpha}$ domain with $0<\alpha\le 1$ .
Let $u\in H^1(\Omega)$ be a solution to \eqref{P} with $\gamma$ and $g$ satisfying the a-priori assumptions stated above. Then  there exists a constant $0<\alpha^{\prime}<1$ and a constant $C>0$  depending on the a-priori data only, such that $u\in C^{1,\alpha^{\prime}}(\bar{\Omega})$,  such that 
\begin{eqnarray}\label{holderimproved}
\|u\|_{C^{1,\alpha^{\prime}}(\bar{\Omega})}\le C \ .
\end{eqnarray}
\end{lemma}
\begin{proof}
Again the proof relies in a slight adaptation of the arguments developed in \cite{S1} based on the Moser iteration technique and by well-known regularity bounds for the Neumann problem \cite[p.667]{ADN}.
\end{proof}

\begin{theorem}\label{localbnds}
Let $\Omega$ be a Lipschitz domain and let $v\in H^{1}\left( \Omega\right) $ be a solution to

\begin{eqnarray}
\Delta v=0  \text{ in }\Omega\text{.}
\
\end{eqnarray}
If its trace $v_{|\partial \Omega}\in H^{1}\left( \partial \Omega\right) $ then $\frac{
\partial v}{\partial \nu }_{|\partial \Omega}\in L^{2}\left( \partial \Omega\right) $
and we have
\begin{eqnarray}\label{normal}
\left \Vert \frac{\partial v}{\partial \nu }\right \Vert _{L^{2}\left(
\partial \Omega\right) }^{2}\leq C\left( \left \Vert \nabla _{T}v\right \Vert
_{L^{2}\left( \partial \Omega\right) }^{2}+\left \Vert v\right \Vert _{H^{1}\left(
\Omega\right) }^{2}\right) \text{.}
\end{eqnarray}
Conversely, if $\frac{\partial v}{\partial \nu }_{|\partial \Omega}\in
L^{2}\left( \partial \Omega\right) $ then $v_{|\partial \Omega}\in H^{1}\left(
\partial \Omega\right) $ and
\begin{eqnarray}\label{tangential}
\left \Vert \nabla _{T}v\right \Vert _{L^{2}\left( \partial \Omega\right) }^{2}\leq
C\left( \left \Vert \frac{\partial v}{\partial \nu }\right \Vert _{L^{2}\left(
\partial \Omega\right) }^{2}+\left \Vert v\right \Vert _{H^{1}\left( \Omega\right) }^{2}\right) \text{.}
\end{eqnarray}
Here $\nabla _{T}v$ denotes the tangential gradient of $v$ on $\partial \Omega$
and $C$ depends on $M,r_{0}$ and $d_0$ only.
\end{theorem}

\begin{proof} These inequalities follow from well-known Rellich's identity \cite{R}. Related estimates were first proven by Jerison and Kenig \cite{JK}. A detailed proof in the present form can be found in \cite[Proposition 5.1]{AMR}. 
\end{proof}

\begin{theorem}\label{h1}
Let $u$ be as in Lemma \ref{regolaritarough}, then
\begin{eqnarray}
\|u\|_{H^1(\partial \Omega)}\le C \ ,
\end{eqnarray}
where $C>0$ depends on the a priori data only.
\end{theorem}
\begin{proof}
The proof is a consequence of \eqref{tangential} in combination with the impedance condition in \eqref{P}, the regularity assumption \eqref{boundg} on $g$ and standard estimates for solution to boundary value problem for the Laplace equation. 
\end{proof}

\section{\normalsize{The inverse problem}}\label{sec5}
\setcounter{equation}{0}
In this section we shall discuss the desired stability estimates. For a sake of exposition we first discuss in Subsection \ref{lb} the case when the boundary $\Gamma_I$ is of Lipschitz class only. While the treatment of the case when $\Gamma_I$ is $C^{1,1}$-smooth will follow in Subsection \ref{cb}.
\subsection{\normalsize{The Lipschitz corroded boundary case}}\label{lb}

\begin{lemma} \label{stimaduale}
Let $u\in H^{1}\left( \Omega\right) \cap C^{0}\left( \overline{\Omega}
\right) $,  be a solution to%
\begin{eqnarray}
\Delta u=0\text{ in }\Omega\text{.}
\end{eqnarray}
We have
\begin{eqnarray}\label{H^{-1}}
\left \Vert \frac{\partial u}{\partial \nu }\right \Vert _{H^{-1}\left(
\Gamma_I\right) }\leq C\left \Vert u\right \Vert _{L^{\infty }\left( \Omega\right) }
\end{eqnarray}
where $C$ depends on $M,r_{0}$ and $d_0$ only.
\end{lemma}
\begin{proof}
By standard result on elliptic boundary value problem, for any $\zeta\in H^1_0(\Gamma_I)$ we  can
consider the unique solution $\varphi \in H^{1}\left(\Omega\right) $ to
the Dirichlet problem
\begin{eqnarray}
\left \{
\begin{array}{c}
\Delta \varphi  =0\text{ in }\Omega \ ,\\
\varphi =\zeta \text{ on }\Gamma_I \ , \\
\varphi =0\text{ on }\Gamma_A \ .
\end{array}
\right.
\end{eqnarray}
Moreover we have
\begin{eqnarray}\label{estim}
\left \Vert \varphi \right \Vert _{H^{1}\left( \Omega\right) }\leq
C\left \Vert \zeta \right \Vert _{H_{00}^{1/2}\left( \Gamma_I\right) }
\end{eqnarray}
with $C>0$ only depending on the a priori data. 
By the Green's identity we have that 
\begin{eqnarray}
\int_{\Gamma_I}\varphi \frac{\partial u}{\partial \nu }=\int_{\partial \Omega}u\frac{\partial \varphi}{\partial \nu} 
\end{eqnarray}
hence
\begin{eqnarray}
\left\vert\int_{\Gamma_I}\zeta \frac{\partial u}{\partial \nu }\right\vert
 \leq \int_{\partial \Omega}\left\vert u\frac{\partial \varphi }{%
\partial \nu }\right\vert
\end{eqnarray}
applying \eqref{normal} to $\varphi $ and taking into account \eqref{holderrough} and \eqref{estim} we get
\begin{eqnarray}
\left \vert \int_{\Gamma_I}\zeta \frac{\partial u}{\partial \nu }
\right \vert &\leq& C\|u\|_{L^{\infty}(\Omega)}\left(\|\zeta\|_{H_0^1(\Gamma_I)} + \|\varphi\|_{H^1(\Omega)} \right)\\
&\leq & C \|u\|_{L^{\infty}(\Omega)}\|\zeta\|_{H_0^1(\Gamma_I)}
\end{eqnarray}
and the thesis follows by duality.
\end{proof}

\begin{theorem}\label{stabcp}
Let $u_i, i=1,2$ be as in Theorem \ref{stabrough}. If for some $\eps$ \eqref{err} holds, then 
\begin{eqnarray}
\|u_1-u_2 \|_{L^{\infty}(\Gamma_I)}\le \eta(\eps)
\end{eqnarray}
where $\eta$ is the modulus of continuity introduced in \eqref{eta}.
\end{theorem}
\begin{proof}
The proof follows by a slight adaptation of the argument developed in Proposition 4.4 in \cite{S1}.
\end{proof}

\begin{corollary}\label{cor}
Let $u_i, i=1,2$ be as in Theorem \ref{stabrough}, then we have that 
\begin{eqnarray}
\left\|\frac{\partial u_1}{\partial \nu }- \frac{\partial u_2}{\partial \nu }\right\|_{H^{-\frac{1}{2}}_{00}(\Gamma_I)}\le \eta(\eps)\ .
\end{eqnarray}
\end{corollary}
\begin{proof}
By interpolation and the impedance condition we have that 
\begin{eqnarray}
\left\|\frac{\partial u_1}{\partial \nu }- \frac{\partial u_2}{\partial \nu }\right\|_{H^{-\frac{1}{2}}_{00}(\Gamma_I)}\le C\left\|\frac{\partial u_1}{\partial \nu }- \frac{\partial u_2}{\partial \nu }\right\|^{\theta}_{H^{-1}(\Gamma_I)}\left\|\gamma_1 u_1- \gamma_2u_2\right\|_{L^2(\Gamma_I)}^{1-\theta}
\end{eqnarray}
where $C>0, 0<\theta<1$ are constants depending on the a priori data only. Finally by Lemma \ref{stimaduale} and Theorem \ref{regolaritarough} we get the thesis.  
\end{proof}

\begin{theorem}\label{lipprop}
Let $u$ be a weak solution to \eqref{P}. For every $ r,\ 0<r<r_1$ and for every $x_0\in \Gamma_{I}^{r_0}$ we have that 
\begin{eqnarray}\label{lippropd}
\int_{\Delta_r(x_0)}u^2\ge \exp(-K r^{-K})
\end{eqnarray}
where $r_1=\min\{\frac{\rho}{2}, r_0, \frac{1}{4},k_1^{\frac{1}{k_2+1}} \}$ and $k_1,k_2, K>0$ only depend on the a priori data. 
\end{theorem}

\begin{proof}
By the local stability estimates for the Cauchy problem discussed in \cite[Theorem 1.7]{ARRV} and the bounds established earlier in Theorem \ref{regolaritarough} and Theorem \ref{h1}, we get that for any $x_0\in \Gamma_{I}^{r_0}$ and any $0<r<r_1$ we have 
\begin{eqnarray}
\|u\|_{L^2(\Gamma_{\frac{r}{2}}(x_0))}\le C (\|u\|_{H^{\frac{1}{2}}(\Delta_r(x_0))} + \|\partial_{\nu}u\|_{H^{-\frac{1}{2}}(\Delta_r(x_0))})^{\delta}(\|u\|_{L^2(\Gamma_r(x_0))})^{1-\delta}
\end{eqnarray}
where $C>0, 0<\delta<1$ are constants depending on the a priori data only. Moreover, by the following interpolation inequality 
\begin{eqnarray}
\|u\|_{H^{\frac{1}{2}}(\Delta_r(x_0))}\le C\|u\|^{\frac{1}{2}}_{L^2(\Delta_r(x_0))}\|u\|^{\frac{1}{2}}_{H^{1}(\Delta_r(x_0))}
\end{eqnarray} 
where $C>0$ depends on the a priori data only, by the a priori bound in Theorem \ref{h1} and the impedance boundary condition we have that
\begin{eqnarray}\label{cpr}
\left(\int_{\Delta_r(x_0)}u^2 \right)^{\frac{\delta}{2}}\ge C \int_{\Gamma_{\frac{r}{2}}(x_0)}u^2 \ .
\end{eqnarray}

Let us consider $\bar{x}\in \Gamma_r(x_0)$ be such that $B_{\frac{r}{8}}(\bar{x})\subset\Gamma_{\frac{r}{2}}(x_0)$. We now recall that using the arguments of Lipschitz propagation of smallness developed in \cite[Proposition 3.1]{MR2}  we have that 
\begin{eqnarray}
\int_{B_{\frac{r}{16}}(\bar{x})}|\nabla u|^2\ge C \exp(-k_1r^{-k_2})\int_{\Omega}|\nabla u|^2
\end{eqnarray}
where $k_1$ and $k_2$ are positive constants depending on the a priori data only.

Combining the standard inequality 
\begin{eqnarray}
\int_{\Omega}|\nabla u|^2 \ge C_1 \|g\|_{H^{-\frac{1}{2}}(\Gamma_A)}
\end{eqnarray}
and the Caccioppoli inequality 
\begin{eqnarray}
\int_{B_{\frac{r}{16}}(\bar{x})}|\nabla u|^2\le C_2{r^{-{2}}}\int_{B_{\frac{r}{8}}(\bar{x})}|u|^2
\end{eqnarray}
where $C_1,C_2>0$ are constants depending on the a priori data only we have that
\begin{eqnarray}
\int_{B_{\frac{r}{8}}(\bar{x})}|u|^2 \ge C r^2 \exp(-k_1r^{-k_2})
\end{eqnarray}
where $C$ is a constant depending on the a priori data only. 

We observe that if $r<\min\{\frac{1}{4}, k_1^{\frac{1}{k_2+1}}\}$ we have that 
\begin{eqnarray}\label{stimabasso}
\int_{B_{\frac{r}{8}}(\bar{x})}|u|^2 \ge C \exp(-2k_1r^{-k_2})
\end{eqnarray}

Moreover, combining the trivial inequality $\int_{\Gamma_{\frac{r}{2}}(x_0)}u^2\ge \int_{B_{\frac{r}{8}}(\bar{x})}u^2$ with \eqref{cpr} we have that 
\begin{eqnarray}
\int_{\Delta_r(x_0)}u^2\ge C \exp(-\frac{4k_1}{\delta}r^{-k_2}) \ . 
\end{eqnarray}
Finally, we observe that it is possible to find a number $K>0$ depending on $C, k_1, k_2, \delta$ only such the thesis follows. 

\end{proof}

\begin{proposition}\label{weight}
Given $M,K>0$, let $w\ge 0$ be a measurable function on $\Gamma_I^{r_0}$ satisfying the conditions
\begin{eqnarray}
\|w\|_{L^{\infty}(\Gamma_I^{r_0})}\le M
\end{eqnarray}
and 
\begin{eqnarray}
\|w\|_{L^2(\Delta_r(x_0))}\ge \exp(-Kr^{-K})\ \ \mbox{for every}\ x\in \Gamma_I^{r_0} \ \mbox{and} \  r\in(0,r_1)
\end{eqnarray}
 where $r_1$ is as in Theorem \ref{lipprop} with $\rho=r_0$.
Let $f\in C^{\alpha}(\Gamma_I^{r_0})$ such that 
\begin{eqnarray}
|f(x) - f(y)|\le E |x-y|^{\alpha}\ \ \mbox{for every}\ x,y\in \Gamma_I^{r_0}\ .
\end{eqnarray}
If 
\begin{eqnarray}
\int_{\Gamma_I^{r_0}}|f|w\le \eps
\end{eqnarray}
then 
\begin{eqnarray}
\|f\|_{L^{\infty}(\Gamma_I^{r_0})}\le E \eta\left(\frac{\eps}{E}\right)
\end{eqnarray}
where $\eta$ satisfies \eqref{eta} with constants only depending on $M,K,r_0,\alpha, k_1, k_2$.
\end{proposition}
\begin{proof}
The proof of such weighted interpolation inequality relies on slight adaptation of the arguments in \cite[Proposition 1]{asv}. 
\end{proof}

\begin{proof}[Proof of Theorem \ref{stabrough}.]

By a standard interpolation result we have that 

\begin{eqnarray}\label{intstan}
\|u_1(\gamma_1-\gamma_2)\|_{L^2(\Gamma_I)}\le C \|u_1(\gamma_1-\gamma_2) \|_{H^1(\Gamma_I)}^{\frac{1}{3}}\|u_1(\gamma_1-\gamma_2) \|_{H_{00}^{-\frac{1}{2}}(\Gamma_I)}^{\frac{2}{3}}
\end{eqnarray}
where $C>0$ is a constant depending on the a priori data only. 

We observe that 
\begin{eqnarray}
\|u_1(\gamma_1-\gamma_2) \|_{H^1(\Gamma_I)}\le \|\gamma_1-\gamma_2\|_{C^{0,1}(\Gamma_I)}\|u_1\|_{H^1(\Gamma_I)}\le C
\end{eqnarray}
where $C>0$ is a constant depending on the a priori data only. 

Moreover by the impedance condition on $\Gamma_I$ it follows that 
\begin{eqnarray}
\|u_1(\gamma_1-\gamma_2) \|_{H_{00}^{-\frac{1}{2}}(\Gamma_I)}\le \left\|\frac{\partial u_1}{\partial \nu}-\frac{\partial u_2}{\partial \nu}  \right\|_{H_{00}^{-\frac{1}{2}}(\Gamma_I)} + C\|u_1-u_2\|_{H_{00}^{-\frac{1}{2}}(\Gamma_I)}
\end{eqnarray}
where $C>0$ is a constant depending on the a priori data only. 

By combining the estimate in Theorem \ref{stabcp} and in Corollary \ref{cor} we obtain 
\begin{eqnarray}
\|u_1(\gamma_1-\gamma_2) \|_{H_{00}^{-\frac{1}{2}}(\Gamma_I)}\le \eta(\eps) \ . 
\end{eqnarray}
Hence by \eqref{intstan} we have that 
\begin{eqnarray}
\|u_1(\gamma_1-\gamma_2)\|_{L^2(\Gamma_I^{r_0})}\le \eta(\eps) \ . 
\end{eqnarray}
The conclusion follows by applying Proposition \ref{weight} with $w=|u_1|$ and $f=(\gamma_1-\gamma_2)^2$.
\end{proof}

\subsection{\normalsize{The $C^{1,1}$-smooth corroded boundary case}}\label{cb}

\begin{theorem}\label{stabcp2}
Let $u_i, i=1,2$ be as in Theorem \ref{stabimproved}. If for some $\eps$, \eqref{err} holds we have that
\begin{eqnarray}
\|u_1-u_2 \|_{C^1(\Gamma_I)}\le \eta(\eps)
\end{eqnarray}
where $\eta$ is given by \eqref{eta}.
\end{theorem}
\begin{proof}
The proof can be achieved along the lines of Proposition 4.4 in \cite{S1} and Theorem 4.2 in \cite{S}.
\end{proof}

\begin{proposition}\label{dip}
Let $\Gamma_I$ be of class $C^{1,1}$ with constants $r_0, M$. Let $u$ be the solution to the problem \eqref{P}, then there exist constants $K_1>0, \bar{r}>0$ depending on the a priori data only, such that for every $x_0\in\Gamma_I^{r_0}$ and every $r\in (0,\bar{r})$ the following holds
\begin{eqnarray}\label{dii}
\int_{\Delta_{2r}(x_0)}u^2\le K_1 \int_{\Delta_{r}(x_0)}u^2 \ . 
\end{eqnarray}
\end{proposition}

\begin{proof}
We provide here a sketch of the proof. 
Let $v\in H^1{(\Omega)}$ be the weak solution to the problem
\begin{equation}\label{P1}
\left\{
\begin{array}
{lcl}
\Delta v=0\ ,& \mbox{in $\Omega$ ,}
\\
\dfrac{\partial v}{\partial \nu}=1\ ,   & \mbox{on $\G_A$ ,}
\\
\dfrac{\partial v}{\partial \nu} + \gamma u =0 \ , & \mbox{on $\G_I$ .}
\end{array}
\right.
\end{equation}
Dealing as in the proof of Lemma 3.3 of \cite{S1} an relying on an iterated used of the Harnack inequality as well as the Giraud's maximum principle, we may infer that there exists a constant $C>0$ depending on the a priori data only such that $v(x)\ge C$ in $\overline{\Omega}$.

It is trivial to check that the function $z=\frac{u}{v}\in H^1({\Omega})$ satisfies 
\begin{equation}\label{P2}
\left\{
\begin{array}
{lcl}
\mbox{div}(v^2\nabla z)=0\ ,& \mbox{in $\Omega$ ,}
\\
v^2\dfrac{\partial z}{\partial \nu} =gv-u \ , & \mbox{on $\G_A$ ,}
\\
v^2\dfrac{\partial z}{\partial \nu} =0 \ , & \mbox{on $\G_I$ .}
\end{array}
\right.
\end{equation}
Let us observe that such change of variable allows us to treat a new boundary problem with an homogeneous Neumann condition on $\Gamma_I$ instead of the Robin one. 
By the arguments due to Adolfsson and Escauriaza in \cite{AdE} (see also \cite[Proposition 3.5]{ABRV}) we have that $u\in H^1(\Omega)$ satisfies the so called doubling inequality at the boundary which can be stated as follows. There exists a radius $\bar{r}$ depending on the a priori data only such that for any $x_0\in \Gamma^{r_0}_I$ the following holds 
\begin{eqnarray}
\int_{\Gamma_{\beta r}(x_0)}z^2\le C \beta^K\int_{\Gamma_{r}(x_0)}z^2
\end{eqnarray}
for every $r,\beta$ such that $\beta>1$ and $0<\beta r<  4\bar{r}$.

Now, we observe that repeating the arguments in Theorem 4.5 and in Theorem 4.6 in \cite{S} and mainly based on well-known stability estimate for the Cauchy problem we can reformulate the above volume doubling inequality at the boundary for the solution $z$ into a surface doubling inequality for the solution $u$. Indeed, we have that there exists a constant $K_1>0$ depending on the a priori data only, such that for any $x_0\in \Gamma^{r_0}_I$ 
and for every $r\in(0,\bar{r})$ the following holds
\begin{eqnarray}
\int_{\Delta_{2r}(x_0)}u^2\le K_1 \int_{\Delta_{r}(x_0)}u^2 \ , 
\end{eqnarray}
and the thesis follows.

\end{proof}

\begin{theorem}\label{lipprop2}
Let $\Gamma_I$ be of class $C^{1,1}$ with constants $r_0, M$.Let $u$ be a weak solution to \eqref{P}. For every $ r,\ 0<r<r_2$ and for every $x_0\in \Gamma_{I}^{r_0}$ we have that 
\begin{eqnarray}\label{lippropd2}
\int_{\Delta_r(x_0)}u^2\ge \frac{1}{K}r^{K}
\end{eqnarray}
where $r_2=\min\{\bar{r}, r_1\}$ and $K>0$ only depends on the a priori data. 
\end{theorem}
\begin{proof}
 Let $x_0\in \Gamma_{I}^{r_0}$. Dealing as in \cite[Remark 4.11]{AMR2}, we have that 
\begin{eqnarray}\label{1p}
\int_{\Delta_{2^{1-j}{r_2}}(x_0)}u^2\le K_1 \int_{\Delta_{2^{-j}{r_2}}(x_0)}u^2 \ , \ \ \  \mbox{for every}\ \ j=2,3, ...
\end{eqnarray}
By iteration over $j$ we get 

\begin{eqnarray}\label{2p}
\int_{\Delta_{\frac{{r_2}}{2}}(x_0)}u^2\le K_1^{j-1} \int_{\Delta_{2^{-j}{r_2}}(x_0)}u^2 \ , \ \ \  \mbox{for every}\ \ j=2,3, ...
\end{eqnarray}

Hence for any $r<\frac{{r_2}}{2}$ and choosing $j$ such that 
\begin{eqnarray}
2^{-j}{r_2}\le r\le 2^{1-j}{r_2}
\end{eqnarray}
and 
\begin{eqnarray}
q=\frac{\log(K_1)}{\log(2)}
\end{eqnarray}
we have that 
\begin{eqnarray}\label{3p}
\int_{\Delta_{{{r}}(x_0)}}u^2\ge \left(\frac{r}{{r_2}}\right)^q\int_{\Delta_{{\frac{{r_2}}{2}}(x_0)}}u^2\ .
\end{eqnarray}
By \eqref{3p} and \eqref{cpr} we find that 
\begin{eqnarray}\label{4p}
\int_{\Delta_{{{r}}(x_0)}}u^2\ge\left(\frac{r}{{r_2}}\right)^q C \left(\int_{\Gamma_{{\frac{{r_2}}{2}}(x_0)}}u^2 \right)^{\frac{1}{\delta}} 
\end{eqnarray}
where $C>0$ is a constant depending on the a priori data only. 

Let $\bar{x}\subset \Gamma_{{\frac{{r_2}}{4}}(x_0)}$ be such that $B_{\frac{r_2}{32}}(\bar{x})\subset \Gamma_{{\frac{{r_2}}{4}}}(x_0)$. By \eqref{stimabasso} with $r=\frac{r_2}{4}$ we have that 
\begin{eqnarray}\label{5p}
\left(\int_{\Gamma_{{\frac{{r_2}}{4}}(x_0)}}u^2 \right)^{\frac{1}{\delta}} \ge C 
\end{eqnarray}
where $C$ is a constant depending on the a priori data only. 
Combining \eqref{4p} and \eqref{5p} we have that 
\begin{eqnarray}\label{6p}
\int_{\Delta_{{{r}}(x_0)}}u^2\ge C \left(\frac{r}{{r_2}}\right)^q 
\end{eqnarray}
where $C>0$ is a constant depending on the a priori data only. 

We conclude by observing that we may find a constant $K>0$ depending on the a priori data only such that the thesis follows. 
\end{proof}

\begin{proposition}\label{weight2}
Given $M,K>0$, let $w\ge 0$ be a measurable function on $\Gamma_I^{r_0}$ satisfying the conditions
\begin{eqnarray}\label{1w}
\|w\|_{L^{\infty}(\Gamma_I^{r_0})}\le M
\end{eqnarray}
and 
\begin{eqnarray}
\|w\|_{L^2(\Delta_r(x_0))}\ge \frac{1}{K}r^{K}\ \ \mbox{for every}\ x\in \Gamma_I^{r_0} \ \mbox{and} \  r\in(0,r_2)
\end{eqnarray}
 where $r_2$ is as in Theorem \ref{lipprop2}.
Let $f\in C^{\alpha}(\Gamma_I^{r_0})$ such that 
\begin{eqnarray}
|f(x) - f(y)|\le E |x-y|^{\alpha}\ \ \mbox{for every}\ x,y\in \Gamma_I^{r_0}\ .
\end{eqnarray}
If 
\begin{eqnarray}
\int_{\Gamma_I^{r_0}}|f|w\le \eps
\end{eqnarray}
then 
\begin{eqnarray}
\|f\|_{L^{\infty}(\Gamma_I^{r_0})}\le C\left(\frac{\eps}{E}\right)^{\delta^{\prime}}
\end{eqnarray}
where $C>0, 0<\delta^{\prime}<1$ are constants only depending on $M,K,r_0,\alpha, r_2$.
\end{proposition}

\begin{proof}
By the bound in \eqref{1w} we have that 

\begin{eqnarray}
\int_{\Delta_r(x)}w^2\ge M^{-1} \frac{r^{2K}}{K^2}\ ,   \ \ \mbox{for every}\ x\in\Gamma_I^{r_0}\ \mbox{and}\ \ r\in(0,r_2) \ . 
\end{eqnarray}

Let now $\bar{x}$ be such that $|f(\bar{x})|=\| f\|_{L^{\infty}(\Gamma_I^{r_0})}$. By the H\"{o}lder regularity of $f$ we have that for every $r>0$ and $x\in\Delta_r(\bar{x})$ the following holds 
\begin{eqnarray}
|f(\bar{x})|\le |f(x)| + E r^{\alpha} \ . 
\end{eqnarray}
Multiplying the above inequality by the weight $w$ and integrating both sides over $\Delta_{r}(\bar{x})$ we obtain that 
\begin{eqnarray}
|f(\bar{x})|\int_{\Delta_r(\bar{x})}w \le \int_{\Delta_r(\bar{x})}w|f| + Er^{\alpha}\int_{\Delta_r(\bar{x})}w\ , 
\end{eqnarray}  
from which we deduce that 
\begin{eqnarray}
\|f\|_{L^{\infty}(\Gamma_I^{r_0})}&\le& \frac{\eps}{\int_{\Delta_r(\bar{x})}w } + Er^{\alpha}\le \\
&\le& \eps M K^2 r^{-2K} + Er^{\alpha}\ . 
\end{eqnarray}
Now minimizing over $r\in(0,r_2)$, the thesis follows with $\delta^{\prime}=\frac{\alpha}{2K+\alpha}$.
\end{proof}

\begin{proof}[Proof of Theorem \ref{stabimproved}.]
By the impedance condition we have that 
\begin{eqnarray}
\|u_1(\gamma_1 -\gamma_2)\|_{L^2(\Gamma_I^{r_0})}\le\left\|\frac{\partial u_1}{\partial \nu}-\frac{\partial u_2}{\partial \nu}  \right\|_{{L^2}(\Gamma_I^{r_0})} + C\|u_1-u_2\|_{{L^2}(\Gamma^{r_0}_I)}
\end{eqnarray}
where $C>0$ is a constant depending on the a priori data only. 

By Theorem \ref{stabcp2} we obtain that 
\begin{eqnarray}
\|u_1(\gamma_1 -\gamma_2)\|_{L^2(\Gamma_I^{r_0})}\le\eta(\eps) \ .
\end{eqnarray}
By applying the Proposition \ref{weight2} the thesis follows with $w=u_1$ and $\lambda=(\gamma_1-\gamma_2)^2$ up to a possible replacement of the constants $C$ and $\vartheta$ in \eqref{eta}.

\end{proof}

We now follow a slightly different strategy in order to prove Theorem \ref{stabimproved}. The main difference is based on the introduction of the notion of Muckenhoupt weights in Proposition \ref{mw}.
\begin{proposition}\label{mw}
Let $\Gamma_I$ be of class $C^{1,1}$ with constants $r_0, M$. Let $u$ be the solution to the problem \eqref{P}, then there exist constant $p>1, A>0$ depending on the a priori data only, such that for every $\Gamma_I^{r_0}$ and every $r\in (0,\bar{r})$ the following holds
\begin{eqnarray}\label{mwi}
\left (\frac{1}{|\Delta_r(x_0)|}\int_{\Delta_r(x_0)}u^2 \right)\left (\frac{1}{|\Delta_r(x_0)|}\int_{\Delta_r(x_0)}u^{-\frac{2}{p-1}} \right)^{p-1}\le A \ .
\end{eqnarray}

\end{proposition}
\begin{proof}
For a detailed proof we refer to Corollary 4.7 in \cite{S}. The main tools of the proof relies on the above mentioned surface doubling inequality \eqref{dii} and the theory of Muckenhoupt weights \cite{CoFe} as well.

\end{proof}

\begin{proof}[Alternative proof of Theorem \ref{stabimproved}.]
Let $x_0\in \Gamma_I^{r_0}$. Let us choose $r=\frac{\bar{r}}{2}$, where $\bar{r}$ is the radius in Proposition \ref{mw}. By the lower bound in \eqref{lippropd2} with $r=\frac{\bar{r}}{2}$ and with $u=u_2$ we have that 
\begin{eqnarray}\label{lblp}
\int_{\Delta_{\frac{\bar{r}}{2}}(x_0)}u_2^2\ge C
\end{eqnarray}
where $C>0$ is a constant depending on the a priori data only. 

Combining \eqref{mwi} and \eqref{lblp}, we have that for every $x_0\in \Gamma_I^{r_0}$ the following holds
\begin{eqnarray}\label{int}
\left(\int_{\Delta_{\frac{\bar{r}}{2}}(x_0)} |u_2|^{-\frac{2}{p-1}}\right)^{p-1}\le C \ , 
\end{eqnarray}
where $C>0$ is a constant depending on the a priori data only. 

Let us now consider $x\in \Delta_{\frac{\bar{r}}{2}}(x_0)$, then by Theorem \ref{stabcp2} and by \eqref{boundgamma} we have that 
\begin{eqnarray}\label{diff}
|\gamma_1(x)-\gamma_2(x)|\le (\gamma_0 +1)\eta(\eps)\frac{1}{|u_2(x)|} \ .
\end{eqnarray}
Denoting with $\beta= \frac{2}{p-1}$ and combining \eqref{int} and \eqref{diff} we find  that 
\begin{eqnarray}\label{stimaint}
\left(\int_{\Delta_{\frac{\bar{r}}{2}}(x_0)} |\gamma_1(x)-\gamma_2(x)|^{\beta}\right)^{\frac{1}{\beta}}\le \eta(\eps) \ .
\end{eqnarray}
By the a priori bound \eqref{boundgamma}, we get that 
\begin{eqnarray}
\|\gamma_1-\gamma_2 \|_{L^2(\Delta_{\frac{\bar{r}}{2}}(x_0))}\le (2\gamma_0)^{1-\frac{\beta}{2}}\left(\int_{\Delta_{\frac{\bar{r}}{2}}(x_0)} |\gamma_1(x)-\gamma_2(x)|^{\beta} \right)^{\frac{1}{2}}
\end{eqnarray}
which in turn combined with \eqref{stimaint} implies that by a possible further replacement of the constants $C, \theta$ in \eqref{eta} we have 
\begin{eqnarray}\label{stabl2}
\|\gamma_1-\gamma_2\|_{L^2(\Delta_{\frac{\bar{r}}{2}}(x_0))}\le \eta(\eps)\ .
\end{eqnarray}
By interpolation we have that 
\begin{eqnarray}
\|\gamma_1-\gamma_2\|_{L^{\infty}(\Delta_{\frac{\bar{r}}{2}}(x_0))}\le C \|\gamma_1-\gamma_2\|^{\frac{1}{2}}_{L^2(\Delta_{\frac{\bar{r}}{2}}(x_0))}\|\gamma_1-\gamma_2\|^{\frac{1}{2}}_{C^{0,1}(\Delta_{\frac{\bar{r}}{2}}(x_0))}
\end{eqnarray}
where $C>0$ is a constant depending on the a priori data only. 

Hence by the a priori bound \eqref{boundgamma} and \eqref{stabl2} we have that by a possible further replacement of the constants $C, \theta$ in \eqref{eta} we have 
\begin{eqnarray}
\|\gamma_1-\gamma_2\|_{L^{\infty}(\Delta_{\frac{\bar{r}}{2}}(x_0))}\le \eta(\eps)
\end{eqnarray}
By a covering argument we finally deduce the thesis. 

\end{proof}

%%%%%%%%%%%%%%%%%%%%%%%%%%%%%%%%%%%%%%%%%%%%%%%%%%%%%%%%%%%%%%%%%%%%%%%%%%%%%%%%%

\end{document}